\documentclass[review,3p,square]{elsarticle}

\usepackage{amssymb}
\usepackage{amsmath}
\usepackage{amsthm}
\usepackage{stmaryrd}
\usepackage{bbm}
\usepackage{enumitem}
\usepackage{mathrsfs}

\usepackage{fancyhdr}
\usepackage{hyperref}
\hypersetup{%
  pdftitle={BSDEs},%
  pdfsubject={BSDEs},%
  pdfauthor={ShengJun FAN, Ying Hu},%
  pdfkeywords={BSDEs},%
  pdfstartview=FitH,%
  CJKbookmarks=true,%
  bookmarksnumbered=true,%
  bookmarksopen=true,%
  colorlinks=true, linkcolor=blue, urlcolor=blue, citecolor=blue, %
}

\usepackage{cleveref}
\crefname{thm}{Theorem}{Theorems}
\crefname{pro}{Proposition}{Propositions}
\crefname{lem}{Lemma}{Lemmas}
\crefname{rmk}{Remark}{Remarks}
\crefname{cor}{Corollary}{Corollaries}
\crefname{dfn}{Definition}{Definitions}
\crefname{ex}{Example}{Examples}
\crefname{section}{Section}{Sections}
\crefname{subsection}{Subsection}{Subsections}

\newcommand{\eps}{\varepsilon}

\newcommand{\To}{\rightarrow}
\newcommand{\as}{{\rm d}\mathbb{P}\times{\rm d} t-a.e.}

\newcommand{\ps}{\mathbb{P}-a.s.}

\newcommand{\F}{\mathcal{F}}
\newcommand{\E}{\mathbb{E}}

\newcommand{\s}{\mathcal{S}}

\newcommand{\mcal}{\mathcal{M}}

\newcommand{\M}{{\rm M}}

\newcommand{\T}{[0,T]}

\newcommand{\Ln}{\mathcal{I\!L}_n^\lambda}
\newcommand{\Lnk}{\mathcal{I\!L}_{n,k}^\lambda}

\newcommand{\R}{{\mathbb R}}

\newcommand {\Dis}{\displaystyle}

\newtheorem{thm}{Theorem}[section]
\newtheorem{lem}[thm]{Lemma}
\newtheorem{pro}[thm]{Proposition}
\newtheorem{rmk}[thm]{Remark}

\newtheorem{ex}[thm]{Example}

\journal{ArXiv}

\begin{document}
\begin{frontmatter}

\title{{Scalar BSDEs of iterated-logarithmically sublinear  generators with integrable terminal values}\tnoteref{found}}
\tnotetext[found]{This work is supported by National Natural Science Foundation of China (Nos. 12171471, 12031009 and 11631004), by Key Laboratory of Mathematics for Nonlinear Sciences (Fudan University), Ministry of Education, Handan Road 220, Shanghai 200433, China; by Lebesgue Center of Mathematics ``Investissements d'avenir" program-ANR-11-LABX-0020-01, by CAESARS-ANR-15-CE05-0024 and by MFG-ANR-16-CE40-0015-01.
\vspace{0.2cm}}


\author[Fan]{Shengjun Fan} \ead{f\_s\_j@126.com}
\author[Hu]{Ying Hu} \ead{ying.hu@univ-rennes1.fr}
\author[Tang]{Shanjian Tang\corref{cor}
} \ead{sjtang@fudan.edu.cn} \vspace{-0.5cm}

\affiliation[Fan]{organization={School of Mathematics, China University of Mining and Technology},
            city={Xuzhou 221116},
            country={China}}

\affiliation[Hu]{organization={Univ. Rennes, CNRS, IRMAR-UMR6625},
            city={F-35000, Rennes},
            country={France}}

\affiliation[Tang]{organization={Department of Finance and Control Sciences, School of Mathematical Sciences, Fudan University},
            city={Shanghai 200433},
            country={China}}

\cortext[cor]{Corresponding author}

\begin{abstract}
We establish a general existence and uniqueness of integrable adapted solutions to scalar backward stochastic differential equations with integrable parameters, where the generator $g$ has an iterated-logarithmic uniform continuity  in the second unknown variable $z$. The result improves our previous one in \cite{FanHuTang2023SCL}.\vspace{0.2cm}
\end{abstract}

\begin{keyword}
backward stochastic differential equation \sep integrable parameters \sep  iterated-logarithmically sublinear generator \sep existence and uniqueness \sep comparison theorem.
\vspace{0.2cm}

\MSC[2010] 60H10\vspace{0.2cm}
\end{keyword}

\end{frontmatter}
\vspace{-0.4cm}

\section{Introduction}
\label{sec:1-Introduction}
\setcounter{equation}{0}

Fix an integer $d\geq 1$ and a real $T>0$. Let $(\Omega, \F, \mathbb{P})$ be a complete probability space equipped with augmented filtration $(\F_t)_{t\in\T}$ generated by a standard $d$-dimensional Brownian motion $(B_t)_{t\in\T}$, and $\F_T=\F$. Consider the following scalar backward stochastic differential equation (BSDE in short):
\begin{equation}\label{eq:1}
  Y_t=\xi+\int_t^T g(s,Y_s,Z_s){\rm d}s-\int_t^T Z_s {\rm d}B_s, \ \ t\in\T,
\end{equation}
where $\xi$ is called the terminal condition, which is an $\F_T$-measurable real-valued random variable, the random field
$$
g:\Omega\times\T\times\R\times\R^{1\times d} \to \R
$$
is called the generator, which is an $(\F_t)$-adapted process for each $(y,z)$, and the pair of $(\F_t)$-adapted processes $(Y_t,Z_t)_{t\in\T}$ is called a solution of \eqref{eq:1}, which takes its values in $\R\times\R^{1\times d}$ such that $\ps$, $t\mapsto Y_t$ is continuous, $t\mapsto |g(t,Y_t,Z_t)|+|Z_t|^2$ is integrable, and satisfies \eqref{eq:1}. Denote by BSDE$(\xi,g)$ the BSDE with the terminal condition $\xi$ and the generator $g$, which are called the parameters of the BSDE.

Nonlinear BSDEs were initially introduced in \cite{PardouxPeng1990SCL}, who established an existence and uniqueness result on adapted solutions of multidimensional BSDEs with Lipschitz continuous generators and square integrable parameters. Since then,  BSDEs have been extensively studied due to their deep connections with various fields such as partial differential equations, mathematical finance, stochastic control and so on. The reader is  refereed to among others \cite{ElKarouiPengQuenez1997MF, Kobylanski2000AP, HuImkeller2005AAP} for more details. In particular, much efforts have been paid on the well-posedness of adapted solutions to BSDEs under various integrability on the parameters and  various growth and/or continuity  of the generator $g$ in the unknown variables $(y,z)$. For instance, some classical results can be found in \cite{LepeltierSanMartin1997SPL, Kobylanski2000AP, BriandHu2006PTRF, BriandHu2008PTRF, DelbaenHuBao2011PTRF, HuTang2015SPA, HuTang2018ECP, BuckdahnHuTang2018ECP, FanHu2019ECP, FanHu2021SPA, FanHuTang2023SPA} and  the references therein.

With adapted solutions of BSDEs, Peng \cite{Peng1997Notes} introduced the notion of conditional $g$-expectation of a square-integrable random variable, which is a  nonlinear extension of the conventional conditional expectation. Since the conventional conditional expectation is defined in the space of integrable  random variables, it is then asked that how to define the conditional $g$-expectation of an only integrable random variable---which entails solution of  BSDEs with only integrable parameters. It has been widely recognized that it is more difficult to solve BSDEs with only integrable parameters than those with $L^p$-integrable parameters ($p>1$). To the best of our knowledge, there are only few discussions on  adapted solution of  BSDEs with integrable parameters. In particular, an existence and uniqueness result is available  in \cite{BriandDelyonHu2003SPA}  in this direction for multidimensional BSDEs, where the generator $g$ is Lipschitz continuous and  grows   in $z$ in a sublinear way (see for example ${\rm (H2S)_\alpha}$ with $\alpha\in (0,1)$ in \cref{sec:2-Statement}). Subsequently  in \cite{FanLiu2010SPL,Fan2016SPA,Fan2018JOTP,XiaoFan2020KM}, new growth and continuity of the generator $g$ in $z$ are given for the  well-posedness of the integrable adapted solutions, such as the uniform continuity and a sublinear growth, the H\"{o}lder continuity and the quasi-H\"{o}lder continuity (see ${\rm (H5S)_\alpha}$ with $\alpha\in (0,1)$ in \cref{sec:2-Statement}). In the last few years, with the test function method in \cite{FanHu2021SPA} and the localization technique in \cite{BriandHu2006PTRF}, we proved in \cite{FanHuTang2023SCL} the existence of integrable solutions of scalar BSDEs with integrable parameters when the generator $g$ satisfies a logarithmic sublinear growth in $z$ (see ${\rm (H2)}_n^{\lambda}$ with $n=1$ and $\lambda>1/2$ in \cref{sec:2-Statement}), and the uniqueness when further $g$  has a logarithmic uniform continuity  in $z$ (see ${\rm (H5)}_n^{\lambda}$ with $n=1$ and $\lambda>1/2$ in \cref{sec:2-Statement}), which extend the above-mentioned growth and continuity  (see (i) of \cref{rmk:2.5} in \cref{sec:2-Statement} for details).

The objective of the present paper is to study  integrable solutions of scalar BSDEs with only integrable parameters under finer assumptions on the generator $g$, as a continuation of our previous work~\cite{FanHuTang2023SCL}.  The main result is formulated in \cref{thm:MainResult} in \cref{sec:2-Statement}. It gives the existence of an $L^1$ solution of a BSDE with integrable parameters when the generator $g$ has an iterated-logarithmic sublinear growth  in $z$ (see ${\rm (H2)}_n^{\lambda}$ with $n\geq 2$ and $\lambda>1/2$ in \cref{sec:2-Statement}), and the uniqueness of  integrable solutions in a proper space when the generator $g$ further has an iterated-logarithmic uniform continuity in $z$ (see ${\rm (H5)}_n^{\lambda}$ with $n\geq 2$ and $\lambda>1/2$ in \cref{sec:2-Statement}), which improves that of~\cite{FanHuTang2023SCL} (see (i) of \cref{rmk:2.3} in \cref{sec:2-Statement} for details). In fact, it gives a sequence of weaker and weaker conditions on the generator $g$ for the existence and uniqueness of the integrable solution of BSDE$(\xi,g)$ under only integrable parameters. For the existence, we first establish a crucial inequality (see \cref{pro:3.2} in \cref{sec:3-Proof}), and then find a proper test function (see \cref{pro:3.4} in \cref{sec:3-Proof}) to apply It\^{o}'s formula to obtain an a priori bound on the first component of adapted solutions to the approximating  BSDEs (see \cref{pro:3.5} in \cref{sec:3-Proof}), and finally utilize the localization technique to obtain the desired solution. For the uniqueness, we establish a general comparison theorem for the integrable solutions to the  BSDEs (see \cref{pro:2.4} in \cref{sec:2-Statement}), where the same a priori estimate technique as above and Theorem 2.1 of \cite{Fan2016SPA} play a key role.

Let us close the introduction by introducing some necessary notations and spaces used in this paper. For $a,b\in \R$, we denote $a\wedge b:=\min\{a,b\}$, $a^+:=\max\{a,0\}$ and $a^-:=-\min\{a,0\}$, and ${\rm sgn}(x):={\bf 1}_{x>0}-{\bf 1}_{x\leq 0}$, where ${\bf 1}_A$ is the indicator function of set $A$. Let ${\bf S}$ be the set of all continuous nondecreasing function $\rho(\cdot):[0,+\infty)\To [0,+\infty)$ with $\rho(0)=0$. For each pair of nonnegative integer $m>k\geq 0$ and each real sequence $\{a_n\}_{n=1}^{+\infty}$, we use the following convention:
$$
\prod\limits_{i=m}^{k} a_i:=1\ \ \ {\rm and}\ \ \ \sum\limits_{i=m}^{k} a_i:=0.\vspace{0.1cm}
$$
Furthermore, for each integer $n\geq 1$, by induction we denote the following function
$$
\ln^{(1)}(x):=\ln x,\ \ x\geq e^{(1)}\ \ {\rm and}\ \  \ln^{(n)}(x):=\ln^{(n-1)}(\ln^{(1)}(x))=\ln^{(1)}(\ln^{(n-1)}(x)), \ \ x\geq e^{(n)},
$$
where
$$
e^{(1)}:=e\ \ {\rm and}\ \ e^{(n)}:=e^{e^{(n-1)}}.
$$
For each $p>0$, let $\s^p$ be the set of $(\F_t)$-adapted continuous real-valued processes $(Y_t)_{t\in\T}$ satisfying
$$\|Y\|_{{\s}^p}:=\left(\E[\sup_{t\in\T} |Y_t|^p]\right)^{{1\over p}\wedge 1}<+\infty,\vspace{0.2cm}$$
and $\mcal^p$ the set of all $(\F_t)$-adapted $\R^{1\times d}$-valued processes $(Z_t)_{t\in\T}$ satisfying
$$
\|Z\|_{\mcal^p}:=\left\{\E\left[\left(\int_0^T |Z_t|^2{\rm d}t\right)^{p/2}\right] \right\}^{{1\over p}\wedge 1}<+\infty.\vspace{0.1cm}
$$
Denote by $\Sigma_T$ the set of all $(\F_t)$-stopping times $\tau$ valued in $\T$. For an $(\F_t)$-adapted real-valued process $(X_t)_{t\in\T}$, if the family $\{X_\tau: \tau\in \Sigma_T\}$ is uniformly integrable, then we say that it is of class (D).

The rest of this paper is organized as follows. In section 2 we state the main result and introduce several remarks and examples to illustrate our theoretical result, and in section 3 we give the proof.

\section{Statement of the main result}
\label{sec:2-Statement}
\setcounter{equation}{0}

We always suppose that $n\geq 1$ is a positive integer, $\alpha\in (0,1)$, $\beta,\lambda\geq 0$ and $\gamma,c>0$ are several nonnegative constants, $\xi$ is a terminal condition satisfying $\E[|\xi|]<+\infty$, and $(f_t)_{t\in \T}$ is an $(\F_t)$-adapted nonnegative process satisfying\vspace{-0.1cm}
$$
\E\left[\int_0^T f_t{\rm d}t\right]<+\infty.
$$
For an  integer $n\ge 1$ and a real number $\lambda\ge 0$, define the function
$$
\Ln(x):=\prod\limits_{i=1}^{n-1}\sqrt{\ln^{(i)}( e^{(n)}+x)} \left(\ln^{(n)}( e^{(n)}+x)\right)^\lambda, \quad x\ge 0.
$$
An equality or inequality between random variables are always understood in the sense of $\ps$. Let us introduce the following  assumptions on the generator $g$.
\begin{enumerate}

\item [(H1)]\label{H1} $\as$, $g(\omega,t,\cdot,\cdot)$ is continuous.
\item [${\rm (H2)}_n^{\lambda}$]\label{H2} $g$ has a one-sided linear growth in $y$ and an iterated-logarithmic sublinear growth in $z$, i.e., $\as$, for each $(y,z)\in \R\times\R^{1\times d}$,
    $$
    {\rm sgn}(y)g(\omega,t,y,z)\leq f_t(\omega)+\beta|y|+ \frac{\gamma |z|}{\Ln(|z|)}.
    $$

\item[(H3)]\label{H3} $g$ has a general growth in $(y,z)$, i.e., there exists a function $h(\cdot)\in {\bf S}$ such that $\as$,
    $$
    |g(\omega,t,y,z)|\leq f_t(\omega)+ h(|y|)+c |z|^2, \quad  \forall (y,z)\in \R\times\R^{1\times d}.
    $$

\item[(H4)]\label{H4} $g$ satisfies an extended monotonicity condition in $y$, i.e., there exists a concave function $\rho(\cdot)\in {\bf S}$ with $\rho(u)>0$ for $u>0$ and $\int_{0^+}\frac{{\rm d}u}{\rho(u)}=+\infty$ such that $\as$,
    $$
    g(\omega,t,y_1,z)-g(\omega,t,y_2,z)\leq \rho(y_1-y_2),\quad \forall  (y_1,y_2,z)\in \R\times\R\times \R^{1\times d} \hbox{ \rm with $y_1> y_2$}.
    $$

\item[${\rm (H5)}_n^{\lambda}$]\label{H5} $g$ has an iterated-logarithmic uniform continuity in $z$, i.e., there is a linearly growing function $\kappa(\cdot)\in {\bf S}$ such that $\as$,
    $$
    |g(\omega,t,y,z_1)-g(\omega,t,y,z_2)|\leq \kappa\left(\frac{|z_1-z_2|}{\Ln(|z_1-z_2|)}\right), \quad\forall  (y,z_1,z_2)\in \R\times \R^{1\times d}\times \R^{1\times d}.\vspace{0.2cm}
    $$
\end{enumerate}

\begin{rmk}\label{rmk:2.1}
We have the following four assertions.
\begin{itemize}
\item [(i)]  Assumptions (H1), (H3) and (H4)  are also required  in \cite{FanHuTang2023SCL}. And,  assumptions ${\rm (H2)}_{1}^{\lambda}$ and ${\rm (H5)}_{1}^{\lambda}$ with  $\lambda\in (1/2,1]$ are the assumptions (H2) and (H5) of \cite{FanHuTang2023SCL},  respectively.  Assumption ${\rm (H2)}_{1}^{0}$ is exactly the one (H1) of \cite{BriandDelyonHu2003SPA}, where $g$ has a one-sided linear growth in $y$ and a linear growth in $z$, and ${\rm (H5)}_{1}^{0}$ is exactly the uniform continuity assumption of $g$ in $z$ of \cite{Fan2016SPA, FanHuTang2023SPA}.

\item [(ii)] For each $n\geq 1$ and $\lambda>1/2$, we have
$$
\mathcal {I\!L}_{n+1}^\lambda (x)\leq K \, \Ln(x), \quad x\ge 0
$$
with a constant $K>0$ depending only on $(n,\lambda)$. Consequently,  for any  $\lambda>1/2$, both conditions ${\rm (H2)}_{n}^{\lambda}$ and ${\rm (H5)}_{n}^{\lambda}$ become weak as the integer $n$ increases.  

\item [(iii)]  The inequality $\ln^{(n)}(e^{(n)})\geq \ln e=1$ is true  for each $n\geq 1$. Both conditions ${\rm (H2)}_{n}^{\lambda}$ and ${\rm (H5)}_{n}^{\lambda}$ become weak as the parameter $\lambda$ decreases in the interval $(\frac{1}{2}, \infty)$. 

\item [(iv)] For each $n\geq 1$ and $k>e^{(n)}$, there is a constant $K>0$ depending only on $(n,k)$ such that
$$
1\leq \frac{\ln^{(n)}(k+x)}{\ln^{(n)}(e^{(n)}+x)}\leq K,\ \ x\geq 0.
$$
Consequently, the constant $e^{(n)}$ appearing in ${\rm (H2)}_{n}^{\lambda}$ and ${\rm (H5)}_{n}^{\lambda}$ can be replaced with a larger number. In addition,  the three  assumptions (H3), (H4) and ${\rm (H5)}_{n}^{\lambda}$ yield ${\rm (H2)}_{n}^{\lambda}$.
\end{itemize}
\end{rmk}

Our  main result  is stated as the following existence and uniqueness theorem.

\begin{thm}\label{thm:MainResult}
Let $n\geq 2$, $\lambda>1/2$ and the generator $g$ satisfy assumptions (H1), ${\rm (H2)}_{n}^{\lambda}$ and (H3). Then BSDE$(\xi,g)$ admits a solution $(Y_t,Z_t)_{t\in\T}$ such that $(Y, Z)\in \s^p\times \M^p$ with each $p\in (0,1)$,  and $Y$ belongs to class (D). And, there exists a constant $C>0$ depending only on $(\beta,\gamma,n,\lambda,T)$ such that\vspace{0.1cm}
\begin{equation}\label{eq:2.1}
|Y_t|\leq |Y_t|+\int_0^t f_s {\rm d}s\leq C\E\left[|\xi|+\int_0^T f_s {\rm d}s\ \middle|\F_t\right]+C,\ \ t\in \T.\vspace{0.2cm}
\end{equation}
Moreover, if assumptions (H4) and ${\rm (H5)}_{n}^{\lambda}$ also hold for the generator $g$, then the solution $(Y, Z)$ with $Y$ being of  class (D) is unique.\vspace{0.2cm}
\end{thm}

\begin{rmk}\label{rmk:2.3}
With respect to \cref{thm:MainResult}, we make the following two remarks.
\begin{itemize}
\item [(i)] In view of  (i) and (ii) of \cref{rmk:2.1},  \cref{thm:MainResult} improves \cite[Theorem 1]{FanHuTang2023SCL}, since it gives a sequence of weaker and weaker conditions on the generator $g$ for the existence and uniqueness of the integrable solution of BSDE$(\xi,g)$ under only integrable parameters. In addition, the condition of $Z\in {\rm M}^p$ for each $p\in (0,1)$ is not required  for the uniqueness in \cref{thm:MainResult} as in
    \cite[Theorem 1]{FanHuTang2023SCL}, while it is assumed  in \cite{BriandDelyonHu2003SPA, FanLiu2010SPL, Fan2016SPA}.

\item [(ii)] Both Theorem 1 of \cite{FanHuTang2023SCL} and \cref{thm:MainResult} show that under assumptions (H1) and (H3) on the generator $g$, the condition ${\rm (H2)}_{n}^{\lambda}$ with $n\geq 1$ and $\lambda>1/2$ is sufficient to guarantee existence of the integrable solution of BSDE$(\xi,g)$ with integrable parameters. And,  \cite{HuTang2018ECP} showed that ${\rm (H2)}_{n}^{\lambda}$ with $n=1$ and $\lambda=0$ is not enough. Assertion (iii) of \cref{rmk:2.1} states that assumption  ${\rm (H2)}_{n}^{\lambda}$ becomes weak as the parameter $\lambda$ decreases. It is still open whether the condition ${\rm (H2)}_{n}^{\lambda}$ with $n\geq 1$ and $\lambda\in (0,1/2]$ is sufficient for the existence of an integrable solution.\vspace{0.2cm}
\end{itemize}
\end{rmk}

A general comparison result on the integrable solutions of BSDEs with integrable parameters is established in the following proposition, which generalizes Proposition 2.5 in \cite{FanHuTang2023SCL} by (i) and (ii) of \cref{rmk:2.1}, and naturally yields the uniqueness part in \cref{thm:MainResult}.

\begin{pro}\label{pro:2.4}
Let $n\geq 2$, $\lambda>1/2$, $\xi$ and $\xi'$ be two terminal conditions, $g$ and $g'$ be two generators, and $(Y_t, Z_t)_{t\in\T}$ and $(Y'_t, Z'_t)_{t\in\T}$ be  a solution of BSDE$(\xi, g)$ and BSDE$(\xi', g')$, respectively. Suppose that $g$ (resp. $g'$) satisfies assumptions (H4) and ${\rm (H5)}_{n}^{\lambda}$ and $(Y-Y')^+$ is of class (D). If $\xi\leq \xi'$ and
\begin{equation}\label{eq:2.2}
{\bf 1}_{Y_t>Y'_t}\left(g(t,Y'_t,Z'_t)-g'(t,Y'_t,Z'_t)\right)\leq 0\ \ \ ({\rm resp.}\  \ {\bf 1}_{Y_t>Y'_t}\left(g(t,Y_t,Z_t)-g'(t,Y_t,Z_t)\right)\leq 0\ ),
\end{equation}
then we have $Y_t\leq Y'_t$  for each $t\in\T$.\vspace{0.2cm}
\end{pro}

Let us further introduce the following assumptions on the generator $g$, which is closely related to assumptions ${\rm (H2)}_n^{\lambda}$ and ${\rm (H5)}_n^{\lambda}$. We would like to mention that some similar assumptions to the following ${\rm (H2S)}_{\alpha}$ and ${\rm (H5S)}_{\alpha}$ have been used in \cite{BriandDelyonHu2003SPA} and \cite{XiaoFan2020KM}, respectively. It will be shown in the following \cref{rmk:2.5} that ${\rm (H5)}_n^{\lambda}$ and ${\rm (H5S)}_{\alpha}$ are respectively equivalent to the following ${\rm (H5')}_n^{\lambda}$ and ${\rm (H5S')}_{\alpha}$.

\begin{enumerate}
\item [${\rm (H2S)}_{\alpha}$]\label{H2S} $g$ has a one-sided linear growth in $y$ and a sublinear growth in $z$, i.e., $\as$, for each $(y,z)\in \R\times\R^{1\times d}$,
    $$
    {\rm sgn}(y)g(\omega,t,y,z)\leq f_t(\omega)+\beta|y|+\gamma |z|^{\alpha}.
    $$

\item[${\rm (H5S)}_{\alpha}$]\label{H5S} $g$ satisfies a quasi-H\"{o}lder continuity condition in $z$, i.e., there exists a function $\bar \kappa(\cdot)\in {\bf S}$ with linear growth such that $\as$, for each $(y,z_1,z_2)\in \R\times \R^{1\times d}\times \R^{1\times d}$,
    $$
    |g(\omega,t,y,z_1)-g(\omega,t,y,z_2)|\leq \bar\kappa\left(|z_1-z_2|^\alpha\right).
    $$

\item[${\rm (H5')}_n^{\lambda}$]\label{H5'} There exists a constant $A>0$ and a function $\tilde\kappa(\cdot)\in {\bf S}$ with linear growth such that $\as$, for each $(y,z_1,z_2)\in \R\times \R^{1\times d}\times \R^{1\times d}$,
    \begin{equation}\label{eq:2.3}
     |g(\omega,t,y,z_1)-g(\omega,t,y,z_2)|\leq \tilde\kappa\left(|z_1-z_2|\right)
    \end{equation}
    and
    \begin{equation}\label{eq:2.4}
    |g(\omega,t,y,z_1)-g(\omega,t,y,z_2)|\leq \frac{A|z_1-z_2|}{\Ln(|z_1-z_2|)}+A.
    \end{equation}

\item[${\rm (H5S')}_{\alpha}$]\label{H5S'} There exists a constant $A>0$ and a function $\tilde\kappa(\cdot)\in {\bf S}$ with linear growth such that $\as$, for each $(y,z_1,z_2)\in \R\times \R^{1\times d}\times \R^{1\times d}$, \eqref{eq:2.3} holds and
    $$
    |g(\omega,t,y,z_1)-g(\omega,t,y,z_2)|\leq A|z_1-z_2|^\alpha +A.
    $$
\end{enumerate}

\begin{rmk}\label{rmk:2.5}
We have the following several remarks.
\begin{itemize}
\item [(i)] For each $\alpha\in (0,1)$, $n\geq 1$ and $\lambda\geq 0$, there exists a constant $K>0$ depending only on $(\alpha,n,\lambda)$ such that
    $$
     |x|^\alpha\leq \frac{K|x|}{\Ln(|x|)},\ \ x\in\R,
    $$
    which means that ${\rm (H2S)}_{\alpha}\Rightarrow {\rm (H2)}_n^{\lambda}$ and ${\rm (H5S)}_{\alpha}\Rightarrow {\rm (H5)}_n^{\lambda}$. In addition, the bigger the $\alpha$, the weaker the assumptions ${\rm (H2S)}_{\alpha}$ and ${\rm (H5S)}_{\alpha}$. In fact, the assertion on ${\rm (H2S)}_{\alpha}$ is obvious. And, it can be easily proved that for each $0<\alpha<\bar\alpha<1$, if the generator $g$ satisfies ${\rm (H5S)}_{\alpha}$ with the function $\bar\kappa(\cdot)$, then it has to satisfy ${\rm (H5S)}_{\bar\alpha}$ with the function $$\tilde\kappa(x):=\bar\kappa(x^{\alpha\over \bar\alpha}){\bf 1}_{0\leq x\leq 1}+\bar\kappa(x){\bf 1}_{x>1}.$$

\item [(ii)] It holds that ${\rm (H5)}_n^{\lambda}\Leftrightarrow {\rm (H5')}_n^{\lambda}$ for each $n\geq 1$ and $\lambda\geq 0$. In fact, since $\ln^{(i)}(e^{(n)})\geq 1$ for each $n\geq 1$ and $i=1,\cdots,n$
    and the function $\kappa(\cdot)$ in ${\rm (H5)}_n^{\lambda}$ is of linear growth, the statement of ${\rm (H5)}_n^{\lambda}\Rightarrow {\rm (H5')}_n^{\lambda}$ is obvious. Conversely, suppose that ${\rm (H5')}_n^{\lambda}$ holds with a constant $A>0$ and a function $\tilde\kappa(\cdot)$. Observe that for each $n\geq 1$ and $\lambda\geq 0$, there exist two constants $K_1,K_2>0$ depending only on $(n,\lambda)$ such that
    $$
    \Ln(x)\leq K_1, \ \ x\in [0,1]
    $$
    and
    $$
    1\leq \frac{K_2 x}{\Ln(x)}, \ \ x\in (1,+\infty).
    $$
    It then follows from \eqref{eq:2.3} and \eqref{eq:2.4} that ${\rm (H5)}_n^{\lambda}$ holds for the generator $g$ with the function
    $$
    \kappa(x):=
    \left\{
    \begin{array}{ll}
    \tilde\kappa(K_1x){\bf 1}_{0\leq x\leq 1}+ \tilde\kappa(K_1) x {\bf 1}_{x>1},& {\rm if}\ \tilde\kappa(K_1)\geq (1+K_2)A;\\
    \frac{(1+K_2)A}{\tilde\kappa(K_1)}\tilde\kappa(K_1x){\bf 1}_{0\leq x\leq 1}+(1+K_2)A x {\bf 1}_{x>1},& {\rm if}\  \tilde\kappa(K_1)<(1+K_2)A.
    \end{array}
    \right.
    $$
    Similarly, it also holds that ${\rm (H5S)}_{\alpha}\Leftrightarrow {\rm (H5S')}_{\alpha}$ for each $\alpha\in (0,1)$. In particular, from the above assertions we can deduce that if the generator $g$ is uniformly continuous and has a bounded growth in $z$, then it must satisfy ${\rm (H5S)}_{\alpha}$ and then ${\rm (H5)}_n^{\lambda}$ for each $\alpha\in (0,1)$, $n\geq 1$ and $\lambda\geq 0$.

\item [(iii)] In view of (ii) and (iii) in \cref{rmk:2.1} and the above (i), the following assertions can be verified.
\begin{itemize}
  \item If $g_1$ and $g_2$ satisfy respectively ${\rm (H5)}_n^{\lambda}$ and ${\rm (H5)}_n^{\bar\lambda}$ with $n\geq 1$ and $0\leq \lambda\leq \bar\lambda$, then anyone of their linear combinations must satisfy ${\rm (H5)}_n^{\bar\lambda}$.
\item If $g_1$ and $g_2$ satisfy respectively ${\rm (H5)}_n^{\lambda}$ and ${\rm (H5)}_m^{\bar\lambda}$ with $1\leq n<m$, $\lambda>1/2$ and $\bar\lambda\geq 0$, then anyone of their linear combinations must satisfy ${\rm (H5)}_m^{\bar\lambda}$.
\item If the generators $g_1$ and $g_2$ satisfy respectively ${\rm (H5S)}_{\alpha}$ and ${\rm (H5S)}_{\bar\alpha}$ with $0<\alpha\leq \bar\alpha<1$, then anyone of their linear combinations satisfy ${\rm (H5S)}_{\bar\alpha}$ and then ${\rm (H5)}_n^{\lambda}$ for each $n\geq 1$ and $\lambda\geq 0$.\vspace{0.2cm}
\end{itemize}
\end{itemize}
\end{rmk}

Finally, let us give two examples to which \cref{thm:MainResult} applies, but none of existing results could.

\begin{ex}\label{ex:2.6}
For each $(\omega,t,y,z)\in\Omega \times\T \times \R\times \R^{1\times d}$, define
$$
g(\omega,t,y,z):=B_t(\omega)-e^{y}\sin^2|z|+\frac{|z|\cos |z|}{\sqrt{\ln(e^8+|z|)}(\ln\ln (e^8+|z|))^{3\over 4}}-|z|^2\sin y.
$$
It is not hard to check that this generator $g$ satisfies assumptions (H1), ${\rm (H2)}_n^{\lambda}$ and (H3) with
$$
f_\cdot=B_\cdot+1,\ \beta=0, \ \gamma=1, \ n=2, \ \lambda=3/4,\  c=1\ \  {\rm and} \ \ h(u)=e^u.
$$
Then, by \cref{thm:MainResult} it can be concluded that BSDE$(\xi,g)$ admits a solution $(Y_t,Z_t)_{t\in \T}$ such that $(Y,Z)\in \s^p\times \M^p$ with each $p\in (0,1)$, and $Y$ is of class (D).\vspace{0.2cm}
\end{ex}

\begin{ex}\label{ex:2.7}
For each $(\omega,t,y,z)\in\Omega \times\T \times \R\times \R^{1\times d}$, define
$$
\bar g(\omega,t,y,z):=y^4{\bf 1}_{y\leq 0}+l(|y|)+\sin|z| +\sqrt{|z|}+|z|^{1\over 3}+\frac{|z|}{\ln (e+|z|)}+\frac{|z|}{\sqrt{\ln (e^8+|z|)}(\ln\ln(e^8+|z|))^{2\over 3}},
$$
where
$$
l(u):=u|\ln u|\ln|\ln u|{\bf 1}_{0\leq u\leq \eps}+l'_-(\eps)(u-l(\eps)){\bf 1}_{u>\eps},\ \ u\in (0,+\infty)
$$
with $\eps>0$ being sufficiently small. Note that for each $x_1,x_2\geq 0$, we have $|F(x_1)-F(x_2)|\leq F(|x_1-x_2|)$ for any concave function $F(\cdot)\in {\bf S}$. It can be verified that for each $x_1,x_2\geq 0$,
$$
|l(x_1)-l(x_2)|\leq l(|x_1-x_2|)\ \ {\rm with}\ \ \int_{0^+} \frac{{\rm d}u}{l(u)}=+\infty,
$$
$$
|\sin x_1-\sin x_2|\leq |x_1-x_2|,\ \ \ |\sin x_1-\sin x_2|\leq 2,\vspace{0.1cm}
$$
$$
\left|\sqrt{x_1}-\sqrt{x_2}\right|\leq \sqrt{|x_1-x_2|},\ \ \ \ \left|x_1^{1\over 3}-x_2^{1\over 3}\right|\leq |x_1-x_2|^{1\over 3},\vspace{0.1cm}
$$
$$
\left|\frac{x_1}{\ln (e+x_1)}-\frac{x_2}{\ln (e+x_2)}\right|\leq \frac{|x_1-x_2|}{\ln (e+|x_1-x_2|)}
$$
and
$$
\begin{array}{ll}
\Dis \left|\frac{x_1}{\sqrt{\ln (e^8+x_1)}(\ln\ln(e^8+x_1))^{2\over 3}}-\frac{x_2}{\sqrt{\ln (e^8+x_2)}(\ln\ln(e^8+x_2))^{2\over 3}}\right|\vspace{0.2cm}\\
\Dis \ \ \leq \frac{|x_1-x_2|}{\sqrt{\ln (e^8+|x_1-x_2|)}\ln\ln(e^8+|x_1-x_2|))^{2\over 3}}.\vspace{0.1cm}
\end{array}
$$
Based on these above observations, by virtue of \cref{rmk:2.5} we can easily prove that this generator $\bar g$ satisfies assumptions (H1), ${\rm (H2)}_n^{\lambda}$, (H3), (H4) and ${\rm (H5)}_n^{\lambda}$ with $\rho(u)=l(u)$, $n=2$ and $\lambda=2/3$. Then, by \cref{thm:MainResult} it can be concluded that BSDE$(\xi,g)$ admits a unique solution $(Y_t,Z_t)_{t\in \T}$ such that $(Y,Z)\in \s^p\times \M^p$ with each $p\in (0,1)$ and $Y$ is of  class (D).
\end{ex}

\section{Proof of the main result}
\label{sec:3-Proof}
\setcounter{equation}{0}

First, we have the following crucial inequality.
\begin{pro}\label{pro:3.1}
Let $p>1$ and $\psi(\cdot):[0,+\infty)\To [1,+\infty)$ be a twice continuously differentiable function such that for each $x\geq 0$, $\psi'(x)>0$, $\psi''(x)<0$,
\begin{equation}\label{eq:3-1}
x\left(\ln \psi(x)\right)'\leq {1\over 4}\wedge (1-\frac{1}{\sqrt{p}}),
\end{equation}
\begin{equation}\label{eq:3-2}
-x\left(\ln \psi'(x)\right)'\leq {3 \over 2},
\end{equation}
and
\begin{equation}\label{eq:3-3}
\psi(\sqrt{p} x\psi(x))\leq \sqrt{p}\psi(x).
\end{equation}
Then, we have
\begin{equation}\label{eq:3-4}
\frac{2xy}{\psi(y)}\leq \frac{p x^2}{\psi^2(x)}+y^2, \quad\forall (x, y)\in [0,\infty)\times [0,\infty).
\end{equation}
\end{pro}

\begin{proof}
For $(x,y)\in [0,+\infty)\times [0,+\infty)$, define the function\vspace{-0.1cm}
\begin{equation}\label{eq:3-5}
f(x,y):=\Dis y^2-\frac{2xy}{\psi(y)}+\frac{px^2}{\psi^2(x)}=\Dis \left(y-\frac{x}{\psi(y)}\right)^2+\frac{px^2}{\psi^2(x)}
-\frac{x^2}{\psi^2(y)}.\vspace{0.1cm}
\end{equation}
Clearly, it suffices to prove that $f(x,y)\geq 0$ for each $x,y\geq 0$. By \eqref{eq:3-5} it is obvious for $y\geq x$. Hence,  we only need to prove that $f(x,y)\geq 0$ for each $x\in (0,+\infty)$ and $y\in [0,x]$.\vspace{0.2cm}

Now, fix arbitrary $x\in (0,+\infty)$ and let $\bar f(y):=f(x,y),\ y\in [0,x]$. Then we have
\begin{equation}\label{eq:3-6}
\bar f'(y)=2y-2x\left(\frac{y}{\psi(y)}\right)'=2y-\frac{2x(\psi(y)-y\psi'(y))}{\psi^2(y)}, \ \ y\in [0,x]
\end{equation}
and
\begin{equation}\label{eq:3-7}
\bar f''(y)=2-2x\left(\frac{y}{\psi(y)}\right)''=2+
\frac{2x\left(2\psi(y)\psi'(y)-2y(\psi'(y))^2+y\psi(y)\psi''(y)\right)}
{\psi^3(y)}, \ \ y\in [0,x].
\end{equation}
It follows from \eqref{eq:3-1} and \eqref{eq:3-2} that
$$
2y(\psi'(y))^2-y\psi(y)\psi''(y)\leq {1\over 2}\psi(y)\psi'(y)+{3\over 2}\psi(y)\psi'(y)=2\psi(y)\psi'(y),
$$
which together with \eqref{eq:3-7} yields that $\bar f''(\cdot)>0$, and then $\bar f(\cdot)$ is a strictly convex function on $[0,x]$. Furthermore, note by \eqref{eq:3-6} that $\bar f'(0)=-2x/\psi(0)<0$ and\vspace{0.1cm}
$$
\bar f'(x)=2x-\frac{2x}{\psi(x)}+\frac{2x^2\psi'(x)}{\psi^2(x)}\geq \frac{2x^2\psi'(x)}{\psi^2(x)}>0.\vspace{0.1cm}
$$
It follows that there exists a unique $y_0\in (0,x)$ such that $\bar f'(y_0)=0$ and
\begin{equation}\label{eq:3-8}
f(x,y)=\bar f(y)\geq \bar f(y_0)=f(x,y_0),\ \ y\in [0,x].\vspace{0.2cm}
\end{equation}

In the sequel, since the function $y\psi(y), y\in [0,+\infty)$ is strictly increasing with its range being $[0,+\infty)$, we can conclude that there exists a unique real $y_1\in (0,+\infty)$ such that\vspace{-0.1cm}
\begin{equation}\label{eq:3-9}
y_1=\frac{x}{\sqrt{p}\psi(y_1)}.\vspace{0.2cm}
\end{equation}
Then, $y_1\in (0,x)$ and it follows from \eqref{eq:3-6}, \eqref{eq:3-9} and \eqref{eq:3-1} that\vspace{0.1cm}
$$
\bar f'(y_1)=2y_1-\frac{2x}{\psi(y_1)}+ \frac{2xy_1\psi'(y_1)}{\psi^2(y_1)}= -\frac{2x\left[\left(1-\frac{1}{\sqrt{p}}\right)\psi(y_1)-y_1\psi'(y_1) \right]}{\psi^2(y_1)}\leq 0.\vspace{0.2cm}
$$
Therefore, $y_1\leq y_0$ and then by \eqref{eq:3-8} and \eqref{eq:3-5} we deduce that
\begin{equation}\label{eq:3-10}
f(x,y)\geq f(x,y_0)=\left(y_0-\frac{x}{\psi(y_0)}\right)^2+\frac{px^2}{\psi^2(x)}
-\frac{x^2}{\psi^2(y_0)}\geq \frac{px^2}{\psi^2(x)}
-\frac{x^2}{\psi^2(y_1)},\ \ \ y\in [0,x].\vspace{0.2cm}
\end{equation}
Finally, it follows from \eqref{eq:3-3} and \eqref{eq:3-9} that
\begin{equation}\label{eq:3-11}
\sqrt{p}\psi(y_1)\geq \psi(\sqrt{p}y_1\psi(y_1))=\psi(x).
\end{equation}
Then, by \eqref{eq:3-10} and \eqref{eq:3-11} we obtain that $f(x,y)\geq 0$ for each $x\in (0,+\infty)$ and $y\in [0,x]$, which is the desired conclusion.
\end{proof}

By virtue of \cref{pro:3.1}, we can establish the following key inequality.

\begin{pro}\label{pro:3.2}
Let $n\geq 1$, $\lambda\geq 0$ and $p>1$. Then, there exists a positive constant $k_{n,\lambda,p}\geq e^{(n)}$ depending only on $(n,\lambda,p)$ such that for each $k\geq k_{n,\lambda,p}$, we have
\begin{equation}\label{eq:3-12}
\frac{2xy}{\Lnk(y)}\leq \frac{px^2}{\left(\Lnk(x)\right)^2}+y^2,\ \ x,y\geq 0,
\end{equation}
where and hereafter,
$$
\Lnk(x):=\prod\limits_{i=1}^{n-1}\sqrt{\ln^{(i)}(k+x)} \left(\ln^{(n)}( k+x)\right)^\lambda.\vspace{0.2cm}
$$
\end{pro}

\begin{proof}
Note that for each $m\geq 2$, the case of $n=m$ and $\lambda=0$ is just the case of $n=m-1$ and $\lambda=1/2$, and the case of $n=1$ and $\lambda=0$ is evident from the basic ineuqality. It suffices to prove the case of $n\geq 1$ and $\lambda>0$. Fix a sufficiently large $k\geq e^{(n)}$ and let
$$
\psi(x):=\Lnk(x)\geq 1,\ \ x\in [0,+\infty).
$$
We prove that for a sufficient large $k$, the function $\psi(\cdot)$ satisfies \eqref{eq:3-1}-\eqref{eq:3-3}. Observe that for each $n\geq 1$ and $x\geq 0$, we have\vspace{0.1cm}
$$
(\ln^{(n)}(k+x))'=\frac{(\ln^{(n-1)}(k+x))'}{\ln^{(n-1)}(k+x)}=\cdots
=\frac{1}{(k+x)\prod\limits_{j=1}^{n-1} \ln^{(j)}(k+x)}.
$$
For each $n\geq 1$ and $x\geq 0$, we can calculate that
$$
\ln \psi(x)={1\over 2}\sum\limits_{i=1}^{n-1}\ln^{(i+1)}(k+x)+\lambda \ln^{(n+1)}(k+x)
$$
and, in view of $k$ being large enough,
\begin{equation}\label{eq:3-13}
\begin{array}{lll}
\Dis 0\leq x\left(\ln \psi(x)\right)'&=& \Dis \frac{x}{k+x}\left\{{1\over 2}\sum\limits_{i=1}^{n-1}\left(\frac{1}{\prod\limits_{j=1}^{i} \ln^{(j)}(k+x)}\right)
+\frac{\lambda}{\prod\limits_{j=1}^{n} \ln^{(j)}(k+x)}\right\}\vspace{0.2cm}\\
&\leq & \Dis \frac{{n-1\over 2}+\lambda}{\ln k}\leq {1\over 4}\wedge (1-\frac{1}{\sqrt{p}}),\vspace{0.1cm}
\end{array}
\end{equation}
which means that the function $\psi(\cdot)$ satisfies \eqref{eq:3-1}.

Furthermore, it is not very difficult to verify that for each $n\geq 1$ and $x\geq 0$,
$$
\psi'(x)=(\ln \psi(x))'\psi(x)=\frac{\left[\sum\limits_{i=1}^{n-1}\left(\prod
\limits_{j=i+1}^{n-1}\ln^{(j)}(k+x)\right)+\frac{2\lambda}{\ln^{(n)}(k+x)}\right]
\psi(x)}{2(k+x)\prod\limits_{j=1}^{n-1} \ln^{(j)}(k+x)}>0
$$
and
$$
\ln \psi'(x)=\ln\psi_0(x)+\ln \psi(x)-\ln 2 -\sum\limits_{i=0}^{n-1} \ln^{(i+1)}(k+x)
$$
with
\begin{equation}\label{eq:3-14}
\psi_0(x):=\sum\limits_{i=1}^{n-1}\left(\prod\limits_{j=i+1}^{n-1} \ln^{(j)}(k+x)\right)+\frac{2\lambda}{\ln^{(n)}(k+x)}\geq \frac{2\lambda}{\ln(k+x)}{\bf 1}_{n=1}+{\bf 1}_{n\geq 2}.\vspace{0.1cm}
\end{equation}
Note that for each $n\geq 1$ and $x\geq 0$, we have with $i=1, \cdots, n-1$,
$$
\left(\prod\limits_{j=i+1}^{n-1}\ln^{(j)}(k+x)\right)'
=\left(\sum\limits_{j=i+1}^{n-1}\ln^{(j+1)}(k+x)\right)'
\prod\limits_{l=i+1}^{n-1}\ln^{(l)}(k+x)
=\frac{
\sum\limits_{j=i+1}^{n-1}\left(\prod\limits_{l=j+1}^{n-1} \ln^{(l)}(k+x)\right)}
{(k+x)\prod\limits_{l=1}^{i}\ln^{(l)}(k+x)}\geq 0,
$$
and then
$$
\psi_0'(x)=\sum\limits_{i=1}^{n-1}\left\{\frac{
\sum\limits_{j=i+1}^{n-1}\left(\prod\limits_{l=j+1}^{n-1} \ln^{(l)}(k+x)\right)}
{(k+x)\prod\limits_{l=1}^{i}\ln^{(l)}(k+x)}\right\}-\frac{2\lambda}{(k+x)
\prod\limits_{i=1}^{n} \ln^{(i)}(k+x)\ln^{(n)}(k+x)}.
$$
From~\eqref{eq:3-14} and \eqref{eq:3-13},  we deduce that for a sufficiently large $k$,
$$
\begin{array}{lll}
-\left(\ln \psi'(x)\right)'&=& \Dis -\frac{\psi'_0(x)}{\psi_0(x)}-\left(\ln \psi(x)\right)'+{1\over k+x}\sum\limits_{i=0}^{n-1}\left(\frac{1}{\prod\limits_{j=1}^{i} \ln^{(j)}(k+x)}\right)\vspace{0.1cm}\\
&\leq & \Dis 0+{1\over k+x}\left(\frac{1}{\ln (k+x)}+\frac{2\lambda}{\ln (k+x)}\right)+0+{1\over k+x}\left(1+\frac{n-1}{\ln (k+x)}\right)\vspace{0.4cm}\\
&\leq & \Dis  {1\over x}\left(1+\frac{n+2\lambda}{\ln k}\right)\leq  {3\over 2x},\ \ x>0,
\end{array}
$$
which yields that the function $\psi(\cdot)$ satisfies \eqref{eq:3-2}.

In the sequel, we prove that the function $\psi(\cdot)$ also satisfies \eqref{eq:3-3}. In fact, fix $n\geq 2$ and $\lambda>0$, and set
$\delta:=p^{\frac{1}{n-1+2\lambda}}>1$. We pick $k$ large enough such that for each $x\geq 0$ and $i=1,\cdots,n$,
\begin{equation}\label{eq:3-15}
k+\sqrt{p} x\psi(x)\leq \sqrt{p}(k+x)\psi(x)\leq (k+x)^\delta\ \ {\rm and}\ \ \delta\ln^{(i)}(k+x)\leq (\ln^{(i)}(k+x))^\delta.
\end{equation}
Then, for each $i=1,\cdots,n$, we have
\begin{equation}\label{eq:3-16}
\ln^{(i)}(k+\sqrt{p} x\psi(x))\leq \delta\ln^{(i)}(k+x),\ \ x\geq 0.
\end{equation}
The last inequality can be proved by induction. In fact, \eqref{eq:3-16} is clear for $i=1$ since it follows from \eqref{eq:3-15} that
$$
\ln^{(1)}(k+\sqrt{p} x\psi(x))\leq \ln^{(1)}[(k+x)^\delta]=\delta\ln^{(1)}(k+x),\ \ x\geq 0.
$$
Now, assume that \eqref{eq:3-16} holds for some $i=l$ with $l\in \{1,\cdots,n-1\}$. Then, in view of \eqref{eq:3-16} and \eqref{eq:3-15} we can deduce that
$$
\begin{array}{lll}
\Dis \ln^{(l+1)}(k+\sqrt{p} x\psi(x))&=& \Dis \ln^{(1)}[\ln^{(l)}(k+\sqrt{p}x\psi(x))]\leq \ln^{(1)}[\delta\ln^{(l)}(k+x)]\leq \ln^{(1)}[(\ln^{(l)}(k+x))^\delta]\\
&=& \Dis \delta\ln^{(1)}[\ln^{(l)}(k+x)]=\delta \ln^{(l+1)}(k+x),\ \ x\geq 0,
\end{array}
$$
which means that \eqref{eq:3-16} also holds for $i=l+1$. Hence, \eqref{eq:3-16} is true for each $i=1,\cdots,n$, and then
$$
\begin{array}{lll}
\Dis \psi(\sqrt{p}x\psi(x))&=&\Dis \prod\limits_{i=1}^{n-1}\sqrt{\ln^{(i)}(k+\sqrt{p}x\psi(x))} \left(\ln^{(n)}(k+\sqrt{p}x\psi(x))\right)^\lambda\\
&\leq & \Dis \delta^{\frac{n-1}{2}+\lambda}\prod\limits_{i=1}^{n-1}\sqrt{\ln^{(i)}(k+x)} \left(\ln^{(n)}(k+x)\right)^\lambda=\sqrt{p}\psi(x),\ \ x\geq 0.
\end{array}
$$
Then, \eqref{eq:3-3} is true for $\psi(\cdot)$. Up to now, we have proved that the function $\psi(\cdot)$ satisfies all conditions in \cref{pro:3.1}, by which the desired inequality \eqref{eq:3-12} follows immediately.
\end{proof}

\begin{rmk}\label{rmk:3.3}
With respect to the above \cref{pro:3.2}, we make the following remarks.
\begin{itemize}
\item [(i)]  The case of $n=1$ and $\lambda\in [0,1]$ of \cref{pro:3.2} is Proposition 3.2 of \cite{FanHuTang2023SCL}.

\item [(ii)] Except from the case of $n=1$ and $\lambda=0$, the inequality \eqref{eq:3-12} does not hold when $p\leq 1$. In fact, let $n\geq 1$, $\lambda>0$ and $k\geq e^{(n)}$ be any constant. Assume that constants $x,y>0$ satisfy
    $$
    y:=\frac{x}{\Lnk(y)}<x.
    $$
Then, in view of \eqref{eq:3-5},
$$
y^2-\frac{2xy}{\Lnk(y)}+\frac{x^2}{\left(\Lnk(x)\right)^2}= \frac{x^2}{\left(\Lnk(x)\right)^2}-\frac{x^2}{\left(\Lnk(y)\right)^2}<0,
$$
 which  immediately  yields the desired assertion. \vspace{0.2cm}
\end{itemize}
\end{rmk}

Now, let $n\geq 2$, $\lambda>1/2$ and $k\geq e^{(n)}$ be sufficiently large and depends only on $(n,\lambda)$ such that the inequality \eqref{eq:3-12} with $p=2$ in \cref{pro:3.2} holds.  A $C^{1,2}$ function $\phi:[0,T]\times [0,+\infty)\to (0,+\infty)$ will be called a test function,  if $\phi_s>0, \phi_x>0, \phi_{xx}>0$, and
\begin{equation}\label{eq:3-17}
\begin{array}{l}
\Dis -\beta\phi_x(s,x)x -\phi_x(s,x)\frac{\gamma |z|}{\Lnk(|z|)}+{1\over 2}\phi_{xx}(s,x)|z|^2+\phi_s(s,x)\geq 0,\vspace{0.2cm}\\
\hspace*{2.5cm} \Dis (s,x,z)\in [0,T]\times [0,+\infty)\times\R^{1\times d}.
\end{array}
\end{equation}
Here and hereafter, $\phi_s$ is the first-order partial derivative of $\phi$ in the first variable, and $\phi_x$ and $\phi_{xx}$ are respectively the first- and second-order partial derivative of $\phi$ in the second variable. Since (in view of \eqref{eq:3-12} with $p=2$)
$$
\begin{array}{lll}
\Dis -\phi_x(s,x)\frac{\gamma |z|}{\Lnk(|z|)}+{1\over 2}\phi_{xx}(s,x)|z|^2
&=& \Dis \phi_{xx}(s,x)\left( -{\gamma\phi_x(s,x)\over \phi_{xx}(s,x)}\frac{|z|}{\Lnk(|z|)}+{1\over 2}|z|^2\right)\vspace{0.2cm}\\
&\geq&  \Dis -\frac{\gamma^2\phi^2_x(s,x)}{\phi_{xx}(s,x)\left(\Lnk\left(\frac{\gamma \phi_x(s,x)}{\phi_{xx}(s,x)}\right)\right)^2},
\end{array}
$$
a $C^{1,2}$ function $\phi(\cdot,\cdot)$ will be a test function if $\phi>0, \phi_s>0, \phi_x>0, \phi_{xx}>0$, and  moreover, for each $(s,x)\in [0,T]\times [0,+\infty)$,
\begin{equation}\label{eq:3-18}
 -\beta\phi_x(s,x)x-\frac{\gamma^2\phi^2_x(s,x)} {\phi_{xx}(s,x)
\left(\Lnk\left(\frac{\gamma \phi_x(s,x)}{\phi_{xx}(s,x)}\right)\right)^2}+\phi_s(s,x)\geq 0.
\end{equation}

In the sequel, we choose the following function\vspace{-0.1cm}
$$
\phi(s,x):= \, (k+x)\left[1-\left(\ln^{(n)}(k+x)\right)^{1-2\lambda}\right] \mu_s, \ \ (s,x)\in [0,T]\times [0,+\infty)
$$
to explicitly solve \eqref{eq:3-18}, where $\mu_s:[0,T]\to (0,+\infty)$ is a nondecreasing and continuously differentiable function to be assigned. First of all, letting further the constant $k\geq e^{(n)}$ depending only on $(n,\lambda,\gamma)$ be large enough, by a simple computation we can obtain that for each $(s,x)\in [0,T]\times [0,+\infty)$,
$$
\phi_x(s,x)=\left[1-\frac{1}{\left(\ln^{(n)}(k+x)\right)^{2\lambda-1}}
\left(1-\frac{2\lambda-1}{\prod\limits_{i=1}^n\ln^{(i)}(k+x)}\right)\right]\mu_s>0,
$$
$$
\phi_{xx}(s,x)= \frac{2\lambda-1}{(k+x)\left(\Lnk(x)\right)^2}
\left(1-\frac{2\lambda-1}{\prod\limits_{i=1}^n\ln^{(i)}(k+x)}-
\frac{1}{\left(\prod\limits_{i=1}^{n}\ln^{(i)}(k+x)\right)^2}\right)\mu_s>0
$$
and
$$
\phi_s(s,x)= (k+x)\left[1-\left(\ln^{(n)}(k+x)\right)^{1-2\lambda}\right]
\mu'_s>0,\vspace{0.2cm}
$$
which yields that for $(s,x)\in [0,T]\times [0,+\infty)$,
\begin{equation}\label{eq:3.19}
{1\over 2}\mu_s\leq \phi_x(s,x)\leq \mu_s,
\end{equation}
\begin{equation}\label{eq:3.20}
\frac{(2\lambda-1)\mu_s}{2(k+x)\left(\Lnk(x)\right)^2}\leq \phi_{xx}(s,x)\leq \frac{(2\lambda-1)\mu_s}{(k+x)\left(\Lnk(x)\right)^2},
\end{equation}
\begin{equation}\label{eq:3.21}
\phi_s(s,x)\geq {1\over 2} (k+x)\mu'_s,
\end{equation}
and then
\begin{equation}\label{eq:3.22}
\frac{\gamma\phi_x(s,x)}{\phi_{xx}(s,x)}\geq {\gamma\over 2(2\lambda-1)}(k+x)\left(\Lnk(x)\right)^2\geq k+x.\vspace{0.1cm}
\end{equation}
Substituting \eqref{eq:3.19}-\eqref{eq:3.22} into the left hand side of \eqref{eq:3-18},  we have\vspace{0.1cm}
$$
\begin{array}{l}
\Dis -\beta\phi_x(s,x)x-\frac{\gamma^2\phi^2_x(s,x)}{\phi_{xx}(s,x)
\left(\Lnk\left(\frac{\gamma \phi_x(s,x)}{\phi_{xx}(s,x)}\right)\right)^2}
+\phi_s(s,x)
\vspace{0.2cm}\\
\Dis \ \ \geq -\beta(k+x)\mu_s-\frac{\gamma^2 \mu_s^2}{\frac{(2\lambda-1)\mu_s}
{2(k+x)\left(\Lnk(x)\right)^2}\left(\Lnk(k+x)\right)^2}+{1\over 2} (k+x)\mu'_s \vspace{0.3cm}\\
\Dis \ \ \geq(k+x)\left[-\left(\beta+{2\gamma^2\over 2\lambda-1}\right)\mu_s+{1\over 2}\mu'_s\right], \ \ (s,x)\in [0,T]\times [0,+\infty).\vspace{0.2cm}
\end{array}
$$
Thus, if we take
$$
\mu_s:=\exp\left[2\left(\beta+{2\gamma^2\over 2\lambda-1}\right)s\right],\ \ s\in [0,T],
$$
then \eqref{eq:3-18} and then \eqref{eq:3-17} holds. \vspace{0.3cm}

Define the function for $k\geq e^{(n)}$, \vspace{0.1cm}
\begin{equation}\label{eq:3-23}
\varphi(s,x):=(k+x)\left[1-\left(\ln^{(n)}(k+x)
	\right)^{1-2\lambda}\right]
\exp\left[2\left(\beta+{2\gamma^2\over 2\lambda-1}\right)s\right],\ (s,x)\in [0,T]\times [0,+\infty).\vspace{0.1cm}
\end{equation}
We have established the following proposition on the test function $\phi$.
\begin{pro}\label{pro:3.4}
Let $n\geq 2$, $\lambda>1/2$ and $k\geq e^{(n)}$ be a sufficiently large constant depending only on $(n,\lambda,\gamma)$ and such that the inequalities~\eqref{eq:3-12} (with $p=2$)  and \eqref{eq:3.19}-\eqref{eq:3.22} are all satisfied.
Then,  we have  for each $(s,x,z)\in [0,T]\times [0,+\infty)\times\R^{1\times d}$,\vspace{0.1cm}
\begin{equation}\label{eq:3-24}
\Dis -\beta\varphi_x(s,x)x -\varphi_x(s,x)\frac{\gamma |z|}{\Lnk(|z|)}+{1\over 2}\varphi_{xx}(s,x)|z|^2+\varphi_s(s,x)\geq 0.\vspace{0.3cm}
\end{equation}
\end{pro}

The following \cref{pro:3.5} gives an a priori estimate for the solution to a BSDE.

\begin{pro}\label{pro:3.5}
Assume that $n\geq 2$, $\lambda>1/2$, the generator $g$ satisfies assumption ${\rm (H2)}_n^{\lambda}$, and $(Y_t,Z_t)_{t\in\T}$ is a solution of BSDE$(\xi,g)$. If the process $(|Y_t|+\int_0^t f_s {\rm d}s)_{t\in\T}$ is of  class (D), then there exists a constant $C>0$ depending only on $(n,\beta,\gamma,\lambda,T)$ such that
\begin{equation}\label{eq:3-25}
|Y_t|\leq |Y_t|+\int_0^t f_s {\rm d}s\leq C\E\left[\left.|\xi|+\int_0^T f_t{\rm d}t\right|\F_t\right]+C,\ \ t\in\T.
\end{equation}
\end{pro}

\begin{proof}
In view of (iv) of \cref{rmk:2.1}, we suppose that the generator $g$ satisfies
assumptions ${\rm (H2)}_n^{\lambda}$  with $\Ln$ being replaced with  $\Lnk$ for  a sufficiently large constant $k\geq e^{(n)}$ (as given in \cref{pro:3.4}). Define
$$
\bar Y_t:=|Y_t|+\int_0^t f_s {\rm d}s\ \ \ \
{\rm and}\ \ \ \ \bar Z_t:={\rm sgn}(Y_t)Z_t,\ \ \ \ t\in \T.
\vspace{0.1cm}
$$
Using It\^{o}-Tanaka's formula, we have\vspace{0.1cm}
$$
\bar Y_t=\bar Y_T+\int_t^T \left({\rm sgn}(Y_s)g(s,Y_s,Z_s)-f_s\right){\rm d}s-\int_t^T \bar Z_s {\rm d}B_s-\int_t^T {\rm d}L_s, \ \ \ t\in\T,
$$
where $L$ is the local time of $Y$ at $0$. Applying It\^{o}-Tanaka's formula with the test function $\varphi$  (see~\eqref{eq:3-23}  for the definition), and noting the assumption ${\rm (H2)}_n^{\lambda}$ (but with $\Ln$ being replaced with  $\Lnk$), we have
$$
\begin{array}{lll}
\Dis {\rm d}\varphi(s,\bar Y_s)
&=&\Dis \varphi_x(s,\bar Y_s)
\left(-{\rm sgn}(Y_s)g(s,Y_s,Z_s)+f_s\right){\rm d}s+\varphi_x(s,\bar Y_s)\bar Z_s {\rm d}B_s\vspace{0.1cm}\\
&&\Dis +\varphi_x(s,\bar Y_s){\rm d}L_s+{1\over 2}\varphi_{xx}(s,\bar Y_s)|Z_s|^2{\rm d}s+\varphi_s(s,\bar Y_s){\rm d}s\vspace{0.2cm}\\
&\geq &\Dis \left[-\varphi_x(s,\bar Y_s)\left(\beta |Y_s|+\frac{\gamma |Z_s|}{\Lnk(|Z_s|)}\right)+{1\over 2}\varphi_{xx}(s,\bar Y_s)|Z_s|^2+\varphi_s(s,\bar Y_s)\right]{\rm d}s\vspace{0.2cm}\\
&& \Dis +\varphi_x(s,\bar Y_s)\bar Z_s {\rm d}B_s, \ \ s\in\T.
\end{array}
$$
Then, since $|Y_s|\leq \bar Y_s$,  we have from \eqref{eq:3-24} in \cref{pro:3.4},
$$
{\rm d}\varphi(s,\bar Y_s)\geq \varphi_x(s,\bar Y_s)\bar Z_s {\rm d}B_s,\ \ s\in \T.
$$
In the sequel, observe from \eqref{eq:3-23} that
$$
{1\over 2}(k+x)\leq \varphi(s,x)\leq k_1(k+x),\ \ (s,x)\in \T\times [0,+\infty),
$$
with
$$
k_1:=\exp\left[2\left(\beta+{2\gamma^2\over 2\lambda-1}\right)T\right].\vspace{0.1cm}
$$
By a similar analysis as that in Proposition 3.4 of \cite{FanHuTang2023SCL},  we conclude that
$$
{1\over 2}(k+\bar Y_t)\leq k_1\E[(k+\bar Y_T)|\F_t],\quad \ t\in \T,
$$
which immediately yields the desired inequality \eqref{eq:3-25}.\vspace{0.2cm}
\end{proof}

We now prove \cref{thm:MainResult} and \cref{pro:2.4}.

\begin{proof}[Proof of \cref{thm:MainResult}]
Assume first that $g$ satisfies assumptions (H1), ${\rm (H2)}_{n}^{\lambda}$ and (H3). With \cref{pro:3.5} in hand, using the localization technique put forward initially in \cite{BriandHu2006PTRF} and following closely the proof of Theorem 2.1 in \cite{FanHuTang2023SCL}, we can construct a solution $(Y_t,Z_t)_{t\in\T}$ of BSDE$(\xi,g)$ such that $(Y,Z)\in \s^p\times\mcal^p$ with each $p\in (0,1)$ and $Y$ is of  class (D). The details are omitted here. Furthermore, suppose that $g$ also satisfies assumptions (H4) and ${\rm (H5)}_{n}^{\lambda}$. The uniqueness part of \cref{thm:MainResult} is a direct consequence of \cref{pro:2.4}, whose proof will be given below. The proof is then complete.\vspace{0.2cm}
\end{proof}

\begin{proof}[Proof of \cref{pro:2.4}]
Let $A$ represent the linear-growth positive constant for the functions $\rho(\cdot)$ and $\kappa(\cdot)$ defined respectively in assumptions (H4) and ${\rm (H5)}_{n}^{\lambda}$. In view of (iv) of \cref{rmk:2.1}, we can suppose that the generator $g$ satisfies ${\rm (H5)}_n^{\lambda}$ with a sufficiently large constant $k\geq e^{(n)}$ instead of $e^{(n)}$, where the constant $k$ depends only on $(n,\lambda, A)$ and can be similarly defined as in \cref{pro:3.4}. In view of assumptions (H4) and ${\rm (H5)}_{n}^{\lambda}$ with $k$ instead of $e^{(n)}$, \cref{pro:3.4} and the proof of \cref{pro:3.5} together with Theorem 2.1 in \cite{Fan2016SPA}, by a similar argument as that in Proposition 2.5 of \cite{FanHuTang2023SCL} we have the desired assertions. The details are omitted here.
\end{proof}

\vspace{0.2cm}




\setlength{\bibsep}{2pt}

\end{document}